\documentclass[reqno]{article}
\usepackage{amsmath, amsfonts, amsthm, amssymb, eucal, indentfirst}

\newtheorem{theorem}{Theorem}
\newtheorem{lemma}[theorem]{Lemma}

\begin{document}

\title{Judicious partitions of uniform hypergraphs}
\author{John Haslegrave
\thanks{Research supported by the Engineering and Physical 
Sciences Research Council}}

\maketitle

\begin{abstract}
The vertices of any graph with $m$ edges may be partitioned into two parts 
so that each part meets at least $\frac{2m}{3}$ edges. Bollob\'as and Thomason 
conjectured that the vertices of any $r$-uniform hypergraph with $m$ edges may 
likewise be partitioned into $r$ classes such that each part meets at least 
$\frac{r}{2r-1}m$ edges. In this paper we prove the weaker statement that,
for each $r\ge 4$, a partition into $r$ classes may be found in which each 
class meets at least $\frac{r}{3r-4}m$ edges, a substantial improvement on previous bounds.
\end{abstract}

\section{Introduction}
The vertices of any graph (indeed, any multigraph) may be partitioned into two parts, each of which meets at 
most two thirds of the edges (\cite{BS93}; it also appears as a problem in \cite{MGT}). An equivalent statement in 
this case is that each part spans at most one third of the edges. These two statements give rise to different 
generalisations when a partition into more than two parts is considered. In this paper we shall only address 
the problem of meeting many edges; the problem of spanning few edges is addressed in \cite{BS99} for the 
graph case and \cite{BS97} for the hypergraph case.

A particularly interesting case occurs when we partition the vertices of an $r$-uniform hypergraph into $r$ 
classes. Bollob\'as and Thomason (see \cite{BRT93}, \cite{BS00}) conjectured that every $r$-uniform hypergraph 
with $m$ edges has an $r$-partition in which each class meets at least $\frac{r}{2r-1}m$ edges.
The author \cite{Hasa} recently proved the conjecture for the case $r=3$. 

The previous best known bound for each $r>3$ was proved by Bollob\'as and Scott \cite{BS00}, who in fact 
obtained a constant independent of $r$. The method they used was to progressively refine a partition by 
repartitioning the vertices in two or three parts to increase the number of parts which met $cm$ edges; in 
doing so they showed that $r$ disjoint parts meeting at least $cm$ edges may be found for $c=0.27$.

Our strategy will be to combine those ideas with the methods used to prove the conjectured bound in the case 
$r=3$ \cite{Hasa}. We shall obtain values of $c_r=\frac{r}{3r-4}$ for $r>3$, and so $c_r \to \frac{1}{3}$ as $r \to \infty$. While this is a significant improvement on previous bounds, it is still some distance from the 
conjectured bounds which approach $\frac{1}{2}$ in the limit.

All results obtained in this paper apply to hypergraphs in which repeated edges are permitted, and in 
fact we shall need this extra generality in an induction step.

\section{New parts from old}
In this section we give two lemmas which show that if we have a partition in which some parts meet many more 
edges than required we may locally refine that partition and increase the number of parts which meet the 
required number of edges. The first of these lemmas was proved in \cite{BS00}; the second is a new result in 
the same spirit. The setting for each lemma is the same. $G$ is a multi-hypergraph, but not necessarily 
uniform; edges may have any number of vertices, and edges of any size may be repeated. $G$ has $m$ edges, but 
the maximum degree of a vertex is less than $cm$ where $c>0$ is an arbitrary constant. For $A,B\subset V$
we write $d(A)$ for the number of edges meeting $A$ and $d(A,B)$ for the number meeting both $A$ and $B$ (we 
shall only use the latter notation where $A$ and $B$ are disjoint).

\begin{lemma}[\cite{BS00}]\label{rutwothree}
Let $c>0$ be a constant and let $G$ be a multi-hypergraph on vertex set $V$ with $m$ edges such that $\Delta(G)<cm$. If $A$ and $B$ are disjoint subsets of $V$, each of which meet at least $2cm$ edges, then there is a partition of $A\cup B$ into three parts, each of which meets at least $cm$ edges.
\end{lemma}
Lemma \ref{rutwothree} was proved in \cite{BS00}; we provide their proof for completeness.
\begin{proof}
We may replace each edge by a subedge if necessary so that no edge meets either $A$ or $B$ in more than 
one vertex. We may then partition $A$ into three parts, $A_1, A_2, A_3$, such that the union of any two
meets at least $cm$ edges: we may do this by taking $A_1$ to be a maximal part meeting fewer than $cm$ edges,
and then dividing $A\setminus A_1$ into two non-empty parts, since any set which strictly contains $A_1$ must 
meet at least $cm$ edges and $A_2\cup A_3=A\setminus A_1$ meets $d(A)-d(A_1)>2cm-cm=cm$ edges. We find a 
similar partition $B=B_1\cup B_2\cup B_3$. It is sufficient to find $i, j$ such that $d(A_i\cup B_j)\ge cm$; 
then $A_i\cup B_j, A\setminus A_i, B\setminus B_j$ will be a suitable partition. We shall show that this is 
always possible.

Since each edge meets $A$ in at most one place then $d(A)=\sum_id(A_i)$ (and likewise $d(B)=\sum_id(B_i)$). 
Similarly, if an edge meets both $A$ and $B$ then we may find unique $i, j$ for which it meets $A_i$ and $B_j$. 
Therefore,
\begin{equation*}
d(A,B)=\sum_{i,j}d(A_i,B_j).
\end{equation*}
Also,
\begin{equation*}
d(A_i\cup B_j)=d(A_i)+d(B_j)-d(A_i,B_j),
\end{equation*}
so
\begin{eqnarray*}
\sum_{i,j}d(A_i\cup B_j) &=& \sum_{i,j}\left(d(A_i)+d(B_j)-d(A_i,B_j)\right)\\
 &=& 3d(A)+3d(B)-\sum_{i,j}d(A_i,B_j)\\
 &=& 3d(A)+3d(B)-d(A,B)\\
 &\ge& 3d(A)+2d(B)\ge 10cm.
\end{eqnarray*}
Since there are nine terms in the LHS, at least one must exceed $cm$. 
\end{proof}

This proof shows more than required: we can always find a partition of $A\cup B$ into three parts, two meeting at least $cm$ edges and the third meeting at least $\frac{10cm}{9}$. However, we cannot always find a partition into three parts, two meeting more than $cm+1$ edges and the third meeting at least $cm$, as seen by considering the case where each of $A$ and $B$ have three vertices, two meeting $cm-1$ edges and one meeting 2 edges. 

A part which does not meet the desired number of edges can also be useful provided we have another part meeting sufficiently many edges to combine with it.

\begin{lemma}\label{rugoodbad}
Let $c>0$ be a constant and let $G$ be a multi-hypergraph on vertex set $V$ with $m$ edges such that $\Delta(G)<cm$. If $A$ and $B$ are disjoint subsets of $V$, with $d(A)\ge 2cm$ and $d(A)+2d(B)\ge 3cm$, then there is a partition of $A\cup B$ into two parts, each of which meets at least $cm$ edges.
\end{lemma}
\begin{proof}
As before, we may replace each edge by a subedge if necessary so that no edge meets either $A$ or $B$ in 
more than one vertex. We may then again partition $A$ into three parts $A_1, A_2, A_3$ such that the union 
of any two meets at least $cm$ edges. It is now sufficient to find $i$ such that $d(A_i\cup B)\ge cm$; then 
$A_i\cup B, A\setminus A_i$ will be a suitable partition. We claim this is always possible.

Again,
\begin{equation*}
d(A,B)=\sum_{i}d(A_i,B)
\end{equation*}
and
\begin{equation*}
d(A_i\cup B)=d(A_i)+d(B)-d(A_i,B),
\end{equation*}
so
\begin{eqnarray*}
\sum_{i}d(A_i\cup B) &=& \sum_{i}\left(d(A_i)+d(B)-d(A_i,B)\right)\\
 &=& d(A)+3d(B)-\sum_{i}d(A_i,B)\\
 &=& d(A)+3d(B)-d(A,B)\\
 &\ge& d(A)+2d(B)\ge 3cm.
\end{eqnarray*}
Since there are three terms in the LHS, at least one must be at least $cm$. 
\end{proof}
In particular, we may find such a partition when $d(A)\ge 2cm$ and $d(B)\ge \frac{cm}{2}$; 
this is the only case in which we shall apply Lemma \ref{rugoodbad}.

\section{The bounds}
Throughout this section $G$ is an $r$-uniform multi-hypergraph on vertex set $V$ with $m$ edges. 
Our goal is to prove that there is a partition of $V$ into $r$ parts such that each part meets at least 
$c_r m$ edges for some suitable constant $c_r$. In order to use Lemmas \ref{rugoodbad} and 
\ref{rutwothree} we need to reduce to the case $\Delta(G)<c_rm$. To that end we apply induction on $r$. If 
some vertex meets at least $c_r m$ edges then we may remove that vertex, replacing each edge of $G$ by a subedge of 
size $r-1$ not containing that vertex, and then use the result for $r-1$ to partition the remaining vertices 
into $r-1$ parts, each meeting at least $c_{r-1} m$ edges; this is sufficient so long as $c_{r-1} \ge c_r$, 
which will be the case. We shall take $c_2=\frac{2}{3}$, so the case $r=2$ is known.

Our plan will be to look at partitions for which $\sum_id(V_i)$ is large, and show that if some parts meet too 
few edges then there are enough parts meeting at least $2c_rm$ edges to allow us to construct a good partition 
by combining parts as above. 

Certainly any partition which is optimal in the sense of maximising $\sum_id(V_i)$ also satisfies the 
local optimality condition that $\sum_id(V_i)$ cannot be increased by moving a single vertex. We shall
consider particularly the yet weaker condition that $\sum_id(V_i)$ cannot be increased by moving a single
vertex into $V_r$. We begin by establishing bounds on partitions satisfying this condition.

\begin{lemma}\label{aaa}
Let $V_1, V_2, \ldots, V_r$ be a partition for which $\sum_i d(V_i)$ cannot be increased by
moving a vertex into $V_r$. Then 
\begin{equation*}
\sum_{i=1}^r d(V_i) \ge (r+1)m-rd(V_r).
\end{equation*}
\end{lemma}
\begin{proof}
For each $i<r$ and each $v\in V_i$, since moving $v$ into $V_r$ does not increase the sum, the number of edges
$e$ such that $e\cap V_i = \{v\}$ must be at least the number of edges $e$ containing $v$ which do not meet 
$V_r$: the first quantity is the decrease in $d(V_i)$ effected by moving $v$ and the second is the increase 
in $d(V_r)$. Thus
\begin{equation*}
\sum_{i\ne r}\sum_{v\in V_i}|\{e:e\cap V_i={v}\}|\le \sum_{i\ne r}\sum_{v\in V_i}|\{e\ni v: e\cap V_r =\emptyset\}|.
\end{equation*}
Since, for $v\ne w$, $\{e:e\cap V_i={v}\}$ and $\{e:e\cap V_i={w}\}$ are disjoint,
\begin{eqnarray*}
\sum_{i\ne r}\sum_{v\in V_i}|\{e:e\cap V_i={v}\}| &=& \sum_{i\ne r}|\bigcup_{v\in V_i}\{e:e\cap V_i={v}\}| \\
 &=& \sum_{i\ne r}|\{e:|e\cap V_i|=1\}|.
\end{eqnarray*}
For $1\le i <r$, let $E_i=\{e:|e\cap V_i|=1\}$. For each edge $e$ write $f(e)$ for the number of parts (including $V_r$)
which meet $e$. If $f(e)<r$ then $|e\cap V_i|>1$ for some $i$, and so $e$ is in at most $f(e)-1$ of the $E_i$; 
trivially if $f(e)=r$ then $e$ is in at most $r-1=f(e)-1$ of the $E_i$. Thus 
\begin{eqnarray*}
\sum_{i\ne r}|\{e:|e\cap V_i|=1\}| &\le& \sum_{e}(f(e)-1) \\
 &=& \sum_{i=1}^r d(V_i)-m.
\end{eqnarray*}
Also,
\begin{equation*}
\sum_{i\ne r}\sum_{v\in V_i}|\{e\ni v: e\cap V_r =\emptyset\}|=r(m-d(V_r)),
\end{equation*}
since each edge not meeting $V_r$ is counted exactly $r$ times in the sum, once for each vertex it contains. 
Combining the above relations, we see that
\begin{equation*}
\sum_{i=1}^r d(V_i)-m \le r(m-d(V_r)),
\end{equation*}
as required.
\end{proof}

The reason for considering the condition that $\sum_id(V_i)$ cannot be increased by moving a single
vertex into $V_r$ is that, as we shall see, it is preserved by moving vertices into $V_r$. If we are able 
to start from a partition which satisfies the condition and in which $V_1,\ldots, V_{r-1}$ are ``good'' (in 
the sense of meeting at least a certain proportion of edges) then we can try to improve the partition by 
moving vertices into $V_r$ while keeping the other parts good. In this way we will either obtain a partition
into $r$ good parts or we will be forced to stop because $V_1,\ldots, V_{r-1}$ are all minimal good sets. We formalise these ideas in the following lemma.

\begin{lemma}\label{aab}
Let $V_1, V_2, \ldots, V_r$ be a partition for which $\sum_i d(V_i)$ is as large as possible, ordered such that
$d(V_1)\ge d(V_2)\ge \cdots \ge d(V_r)$, and let $c>0$ be a constant. If $d(V_{r-1})\ge cm$ then either there 
exists a partition $W_1, W_2, \ldots, W_r$, with $W_i\subseteq V_i$ for $1\le i\le r-1$ 
(and so $W_r\supseteq V_r$), such that each part meets at least $cm$ edges, or there exists a partition 
$W_1, W_2, \ldots, W_r$, again with $W_i\subseteq V_i$ for $1\le i\le r-1$,
such that, for each $i\ne r$, $d(W_i)\ge cm$ but $d(W_i\setminus \{w\})< cm$ for any $w\in W_i$, and also
\begin{equation*}
\sum_{i=1}^{r-1} d(W_i) > (r+1)(m-cm).
\end{equation*}
\end{lemma}
\begin{proof}
Suppose $U_1, U_2, \ldots, U_r$ is a partition with $U_i\subseteq V_i$ for each $i<r$ and so $U_r\supseteq V_r$. 
For any $v\notin U_r$, say $v\in U_i$, let $U'_i=U_i\setminus\{v\}$, $U'_r=U_r\cup\{v\}$, $V'_i=V_i\setminus\{v\}$ 
and $V'_r=V_r\cup\{v\}$. Then
\begin{eqnarray*}
d(U_i)-d(U'_i)&=&|\{e:e\cap U_i=\{v\}\}| \\
 &\ge&|\{e:e\cap V_i=\{v\}\}| \\
 &=&d(V_i)-d(V'_i)
\end{eqnarray*}
and
\begin{eqnarray*}
d(U'_r)-d(U_r)&=&|\{e\ni v:e\cap U_r=\emptyset\}| \\
 &\le&|\{e\ni v:e\cap V_r=\emptyset\}| \\
 &=&d(V'_r)-d(V_r)
\end{eqnarray*}
so
\begin{eqnarray*}
d(U'_i)+d(U'_r) &\le& \left(d(U_i)+d(V'_i)-d(V_i)\right)+\left(d(V'_r)-d(V_r)+d(U_r)\right) \\
 &=&d(U_i)+d(U_r)+\left(d(V'_i)+d(V'_r)-d(V_i)-d(V_r)\right) \\
 &\le&d(U_i)+d(U_r),
\end{eqnarray*}
meaning that $U_1, U_2, \ldots, U_r$ also satisfies the condition that $\sum_i d(U_i)$ cannot be increased by moving a vertex into $U_r$.

For each $i<r$, then, let $W_i$ be a minimal subset of $V_i$ satisfying $d(W_i)\ge cm$, and let $W_r=V\setminus\bigcup_{i<r}W_i$. 
If $d(W_r)\ge cm$ then this is a suitable partition with each part meeting at least $cm$ edges; if not then Lemma \ref{aaa} ensures that
\begin{eqnarray*}
\sum_{i=1}^{r-1} d(W_i) &\ge& (r+1)m-rd(W_r)-d(W_r) \\
 &>&(r+1)(m-cm),
\end{eqnarray*}
as required. 
\end{proof}

We are now ready to prove the main result. We shall show that if we start from a partition maximising $\sum_id(V_i)$ then either we have enough elbow room to obtain a good partition by repeated application of Lemmas \ref{rutwothree} and \ref{rugoodbad} or we may use Lemma \ref{aab} to obtain a good partition.

\begin{theorem}\label{rumain}
Let $c_2=\frac{2}{3}$, $c_3=\frac{5}{9}$ and $c_r=\frac{r}{3r-4}$ for $r>3$. If $G$ is an $r$-uniform 
multi-hypergraph with $m$ edges there is a partition of the vertex set into $r$ parts with each part meeting at 
least $c_r m$ edges.
\end{theorem}
\begin{proof}
We use induction on $r$; the case $r=2$ is known. For $r>2$, if some vertex $v$ meets at least 
$c_r m$ edges then we may, by replacing each edge with a subedge of size $r-1$ not containing $v$ and applying 
the $r-1$ case, find a partition of the other vertices into $r-1$ parts each meeting at least $c_{r-1}m > c_rm$ 
edges. Together with $\{v\}$, this is a suitable $r$-partition. Thus we may assume that no vertex meets $c_r m$ 
edges.

Now let $V_1, V_2, \ldots, V_r$ be a partition for which $\sum_i d(V_i)$ is as large as possible, ordered such 
that $d(V_1)\ge d(V_2)\ge \cdots \ge d(V_r)$. If $d(V_r)\ge c_r m$ then we are done, so we may assume 
$d(V_r)< c_r m$. We consider three cases based on the values of $d(V_{r-1})$ and $d(V_r)$.

\vspace{5pt}
\textbf{Case 1.} $d(V_{r-1})\ge c_r m$.

By Lemma \ref{aab}, either we have a good partition or we may find a partition $W_1, W_2, \ldots, W_r$ such that, for each
$i<r$, $d(W_i)\ge c_r m$ but $W_i$ is minimal for this property, and further that $\sum_{i=1}^{r-1} d(W_i) \ge (r+1)(m-c_r m)$.
Suppose the partition found is of the latter type. For each $i<r$, since $\Delta(G)< c_r m$, $|W_i|>1$. Let $v, w$ be two 
vertices in $W_i$; by minimality of $W_i$ the number of edges meeting $W_i$ only at $v$ is more than $d(W_i)-c_r m$, and so 
is the number meeting $W_i$ only at $w$. These sets of edges are disjoint from each other and from the set of edges 
meeting $W_i$ in more than one vertex; thus, writing $d_2(X)$ for the number of edges meeting $X$ in more than one vertex,
\begin{equation*}
d(W_i) > 2(d(W_i)-c_r m)+d_2(W_i),
\end{equation*}
i.e.
\begin{equation*}
d(W_i)+d_2(W_i)<2c_r m.
\end{equation*}
However, each edge not meeting $W_r$ meets at least one other part at more than one vertex, so
\begin{equation*}
\sum_{i=1}^{r-1}d_2(W_i)>m-c_r m.
\end{equation*}
Combining this with the bound on $\sum_{i=1}^{r-1} d(W_i)$ gives
\begin{equation*}
\sum_{i=1}^{r-1}(d(W_i)+d_2(W_i))>(r+2)(m-c_r m),
\end{equation*}
so for some $i<r$
\begin{equation*}
d(W_i)+d_2(W_i)>\frac{r+2}{r-1}(m-c_r m).
\end{equation*}
Consequently, using our upper bound on $d(W_i)+d_2(W_i)$,
\begin{eqnarray*}
&&\frac{r+2}{r-1}(m-c_r m) < 2c_r m \\
&\Rightarrow&(r+2)(1-c_r) < 2c_r(r-1) \\
&\Rightarrow& r+2 < 3rc_r
\end{eqnarray*}
and so $c_r > \frac{r+2}{3r}$. 

For $r=3$, $c_r=\frac{5}{9}=\frac{r+2}{3r}$; for $r\ge 4$,
\begin{eqnarray*}
c_r &=& \frac{r}{3r-4} \\
&=& \frac{r+2}{3r}-\frac{2-8r}{3r(3r-4)} \\
&\le& \frac{r+2}{3r};
\end{eqnarray*}
in either case a contradiction is obtained (from our assumption that we did not get a good partition from Lemma \ref{aab}). (End of Case 1.)
\vspace{1.5ex}

Note that, using Lemma \ref{aaa}, if $d(V_r) < c_r m$ then 
\begin{eqnarray*}
d(V_{r-1})-c_r m&=&\sum_{i=1}^{r} d(V_i)-\sum_{i=1}^{r-2} d(V_i)-d(V_r)-c_rm \\
&>& \sum_{i=1}^{r} d(V_i)-(r-2)m-2c_r m\\
&\ge& (r+1)m-rd(V_r)-(r-2)m -2c_r m\\
&>& (r+1)m-rc_rm-(r-2)m -2c_r m\\
&=& 3m-(r+2)c_r m,
\end{eqnarray*}
so for $r=3$, since $c_r < \frac{3}{r+2}$, Case 1 is the only possible case. For the remaining cases, then, we assume $r\ge 4$.

\vspace{5pt}
\textbf{Case 2.} $c_r m > d(V_{r-1}) \ge d(V_r) \ge \frac{c_rm}{2}$.

Suppose that $k$ parts meet at least $2c_r m$ edges, $l$ meet fewer than $c_r m$, and the remaining $r-k-l$ meet at least $c_r m$
but fewer than $2c_r m$. Using Lemma \ref{rugoodbad} we may combine a part meeting at least $2c_r m$ edges and a part meeting at least
$\frac{c_rm}{2}$ to produce two parts meeting at least $c_r m$; we know, since $V_r$ meets fewest edges, that each part meets at
least $\frac{c_rm}{2}$ and so we can obtain a good partition provided $k\ge l$. We shall show that this must be so.

Since $r \ge 4$, $2c_r\le 1$, and since $c_r m > d(V_{r-1})$, $l\ge 2$. So
\begin{eqnarray*}
\sum_{i=1}^rd(V_i) &<& 2c_r m(r-k-l)+c_r ml+mk \\
&=& 2c_r mr -c_r lm + (1-2c_r)km.
\end{eqnarray*}
However, by Lemma \ref{aaa},
\begin{equation*}
\sum_{i=1}^r d(V_i) > (r+1)m-rc_r m,
\end{equation*}
so
\begin{equation*}
2rc_r -lc_r + k(1-2c_r) > (r+1)-rc_r;
\end{equation*}
however, if $k<l$ then
\begin{eqnarray*}
2rc_r -lc_r + k(1-2c_r) &\le& 2rc_r-lc_r+(l-1)(1-2c_r) \\
&=& (2r-1)c_r+(l-1)(1-3c_r)\,.
\end{eqnarray*}
Since $l\ge 2$ and $1-3c_r<0$,
\begin{eqnarray*}
(2r-1)c_r+(l-1)(1-3c_r)&\le& (2r-4)c_r+1 \\
&=& r+1-rc_r
\end{eqnarray*}
(the final equality follows since $c_r=\frac{r}{3r-4}$), and a contradiction is obtained.

\vspace{5pt}
\textbf{Case 3.} $c_r m > d(V_{r-1})$ and $\frac{c_rm}{2} > d(V_r)$.

Suppose that $j$ parts meet at least $2c_r m$ edges, $k$ meet at least $\frac{c_rm}{2}$ but fewer than $c_r m$, 
$l$ meet fewer than $\frac{c_rm}{2}$, and the remaining $r-j-k-l$ meet at least $c_r m$ but fewer than $2c_r m$. 
Using Lemma \ref{rugoodbad} $k$ times and Lemma \ref{rutwothree} $l$ times, we may find $r$ disjoint parts which meet at least $c_r m$ edges
provided that $j\ge k+2l$ (each application of Lemma \ref{rugoodbad} requires one part which meets $2c_r m$ edges and each application of Lemma \ref{rutwothree} requires two). We shall show that this must be so.

From Lemma \ref{aaa}, since $d(V_r) < \frac{c_rm}{2}$ and $r\ge 4$, 
\begin{eqnarray*}
\frac{1}{r}\sum_{i=1}^r \frac{d(V_i)}{m} &>& \frac{r+1}{r}-\frac{c_r}{2} \\
&=& 2c_r+\frac{r+1}{r}-\frac{5c_r}{2} \\
&=& 2c_r+\frac{r^2-2r-8}{r(6r-8)} \\
&\ge& 2c_r;
\end{eqnarray*}
similarly, since $c_r\le\frac{1}{2}$,
\begin{eqnarray*}
\frac{1}{r}\sum_{i=1}^r \frac{d(V_i)}{m} &>& \frac{r+1}{r}-\frac{c_r}{2} \\
&>& \frac{3}{4} \\
&\ge& \frac{2}{3}+\frac{c_r}{6};
\end{eqnarray*}
so, since $c_r>\frac{1}{3}$ for every $r$, Lemma \ref{rulast}, following, ensures that $j\ge k+2l$.
\end{proof}

We now give the result needed to fill in the gap.

\begin{lemma}\label{rulast}
Let $\frac{1}{3}\le c\le \frac{1}{2}$. If we have a finite, non-empty collection of numbers in $[0,1]$, whose 
arithmetic mean $a$ is at least $\max\left(2c,\frac{2}{3}+\frac{c}{6}\right)$, of which $j$ are at least $2c$, 
$k$ are at least $\frac{c}{2}$ but less than $c$, $l$ are less than $\frac{c}{2}$, and the rest are at least 
$c$ but less than $2c$, then $j\ge k+2l$. 
\end{lemma}
\begin{proof}
We proceed by induction on $k+l$; if $k=l=0$ then there is nothing to prove. If $k\ge 1$ then one of our 
numbers is less than the mean, so another must exceed the mean; since the mean of those 
two numbers is at most $\frac{1+c}{2}\le 2c \le a$, either there are no other numbers (in which case we are 
done) or the mean of the remaining numbers is at least $a$ and we are done by induction.
If $l\ge 1$ then one number is less than $\frac{c}{2}$; if also $j\le 1$ then the total is less than 
$1+\frac{c}{2}+(n-2)2c<2cn$, so there must be two numbers which are at least $2c$. Since the mean of these 
three numbers is at most $\frac{2}{3}+\frac{c}{6}\le a$, either there are no other numbers (in which case we 
are done) or the mean of the remaining numbers is at least $a$ and again we are done by induction.
\end{proof}
The value of $c_r$ which we use is tight only in the second case of Theorem \ref{rumain} (for $r>4$); if we were to use instead some $c'_r > c_r$ it would be feasible for two of our parts to meet just under $c'_r$ edges while the remaining parts each met just under $2c'_r$ edges, giving us no immediate way to use Lemmas \ref{rutwothree} and \ref{rugoodbad}. To improve further on these bounds using a similar method, then, we might seek to prove analogues of Lemmas \ref{rutwothree} and \ref{rugoodbad} for parts which meet more than $\beta cm$ for some $1<\beta<2$. To do this, however, we would need a way to impose some stronger assumption on the degrees of vertices than simply $\Delta(G)<cm$.

\end{document}